\documentclass[a4paper]{amsart}

\usepackage{a4wide,amsfonts,amssymb,amsmath,color}
\usepackage[english]{babel}
\usepackage{ifthen}
\usepackage{hyperref} 
\hypersetup{hidelinks}
\usepackage[nameinlink]{cleveref} 
\usepackage{mathtools} 
\usepackage{upgreek} 
\usepackage{pifont} 
\usepackage{enumitem} 
\newcommand{\myroman}[1]{\textbf{\textup{(\roman{#1})}}} 
\allowdisplaybreaks

\newtheorem{lem}{Lemma}[section]
\newtheorem{defi}[lem]{Definition}
\newtheorem{theo}[lem]{Theorem}
\newtheorem{cor}[lem]{Corollary}
\newtheorem{rem}[lem]{Remark}
\newcommand{\IfTrueThen}[2]{\ifthenelse{\equal{#1}{}}{}{#2}}
\newcommand{\ol}{\overline}
\def\equi{\quad\Leftrightarrow\quad}
\def\cpt{\overset{\mathsf{cpt}}{\hookrightarrow}}
\def\reals{\mathbb{R}}

\def\tdom{I}

\def\ga{\Gamma}
\def\gat{\ga_{\!0}}
\def\gan{\ga_{\!1}}
\def\om{\Omega}
\def\mcH{\mathcal{H}}
\def\sfL{\mathsf{L}}
\def\sfH{\mathsf{H}}
\def\sfC{\mathsf{C}}
\newcommand{\Harm}[2]{\mcH\IfTrueThen{#1}{^{#1}}\IfTrueThen{#2}{_{#2}}}
\renewcommand{\L}[2]{\sfL\IfTrueThen{#1}{^{#1}}\IfTrueThen{#2}{_{#2}}}
\renewcommand{\H}[2]{\sfH\IfTrueThen{#1}{^{#1}}\IfTrueThen{#2}{_{#2}}}
\newcommand{\hH}[2]{\widehat{\sfH}\IfTrueThen{#1}{^{#1}}\IfTrueThen{#2}{_{#2}}}
\newcommand{\C}[2]{\sfC\IfTrueThen{#1}{^{#1}}\IfTrueThen{#2}{_{#2}}}

\newcommand{\dx}[1][x]{\,\mathrm{d}#1}
\newcommand{\eps}{\varepsilon}
\DeclareMathOperator{\A}{A}
\DeclareMathOperator{\T}{T}
\DeclareMathOperator{\p}{\partial}
\DeclareMathOperator{\grad}{\nabla}
\DeclareMathOperator{\rot}{curl}
\DeclareMathOperator{\dive}{div}
\renewcommand{\div}{\dive}
\DeclareMathOperator{\dom}{dom}
\DeclareMathOperator{\ran}{ran}
\DeclarePairedDelimiter{\norm}{\lVert}{\rVert}
\DeclarePairedDelimiterX{\scp}[2]{\langle}{\rangle}{#1,#2}

\title[Compact Embedding for the div-curl System]
{A Compactness Result for the div-curl System with\\ 
Inhomogeneous Mixed Boundary Conditions\\
for Bounded Lipschitz Domains\\
and Some Applications}
\author{Dirk Pauly}
\address{Fakult\"at f\"ur Mathematik,
Universit\"at Duisburg-Essen, Germany}
\email[Dirk Pauly]{dirk.pauly@uni-due.de}

\author{Nathanael Skrepek}
\address{Fakult\"at f\"ur Mathematik und Naturwissenschaften,
Bergische Universit\"at Wuppertal, Germany}
\email[Nathanael Skrepek]{skrepek@uni-wuppertal.de}
\date{\today}
\thanks{We gratefully acknowledge ISem23
(23rd Internet Seminar 2019/2020, Evolutionary Equations,
\url{http://www.mat.tuhh.de/isem23}) 
for providing the platform to start this research.
The second author has received funding from the European Union's Horizon 2020 
research and innovation programme under the Marie Sklodowska-Curie grant agreement No 765579}

\begin{document}

\def\titlerepude{\sf A Compactness Result for the div-curl System with\\ 
Inhomogeneous Mixed Boundary Conditions\\
for Bounded Lipschitz Domains\\
and Some Applications}
\def\authorrepude{Dirk Pauly \& Nathanael Skrepek}
\def\daterepdue{March 11, 2021}
\def\reportudemathyesno{no}
\def\reportudemathnumber{SM-UDE-825}
\def\reportudemathyear{2021}
\def\reportudematheingang{\daterepdue}
\newcommand{\preprintudemath}[5]{
\thispagestyle{empty}
\begin{center}\normalsize SCHRIFTENREIHE DER FAKULT\"AT F\"UR MATHEMATIK\end{center}
\vspace*{5mm}
\begin{center}#1\end{center}
\vspace*{5mm}
\begin{center}by\end{center}
\vspace*{0mm}
\begin{center}#2\end{center}
\vspace*{5mm}
\normalsize 
\begin{center}#3\hspace{69mm}#4\end{center}
\newpage
\thispagestyle{empty}
\vspace*{210mm}
Received: #5
\newpage
\addtocounter{page}{-2}
\normalsize
}
\ifthenelse{\equal{\reportudemathyesno}{yes}}
{\preprintudemath{\titlerepude}{\authorrepude}{\reportudemathnumber}{\reportudemathyear}{\reportudematheingang}}
{}


\begin{abstract}
For a bounded Lipschitz domain with Lipschitz interface we show 
the following \emph{compactness theorem}:
Any $\L{2}{}$-bounded sequence of vector fields
with $\L{2}{}$-bounded rotations and $\L{2}{}$-bounded divergences
as well as $\L{2}{}$-bounded tangential traces on one part of the boundary 
and $\L{2}{}$-bounded normal traces on the other part of the boundary,
contains a strongly $\L{2}{}$-convergent subsequence.
This generalises recent results for homogeneous mixed boundary conditions
in~\cite{bauerpaulyschomburg2016a,bauerpaulyschomburg2019a}.
As applications we present a related \emph{Friedrichs/Poincar\'e type estimate},
a \emph{div-curl lemma}, and show that the Maxwell operator with mixed 
tangential and impedance boundary conditions (Robin type boundary conditions)
has \emph{compact resolvents}.
\end{abstract}


\maketitle
\tableofcontents


\section{Introduction}


Let $\om\subset\reals^{3}$ be open with boundary $\ga$,
composed of the boundary parts $\gat$ (tangential) and $\gan$ (normal).
In~\cite[Theorem 4.7]{bauerpaulyschomburg2016a}
the following version of Weck's selection theorem has been shown:

\begin{theo}[compact embedding for vector fields with homogeneous mixed boundary conditions]%
\label{cptvechomotheo}
Let $(\om,\gat)$ be a bounded strong Lipschitz pair and let $\eps$ be admissible. Then
\[
  \H{}{\gat}(\rot,\om)\cap\eps^{-1}\H{}{\gan}(\div,\om)\cpt\L{2}{}(\om).
\]
\end{theo}

Here, $\cpt$ denotes a compact embedding, and 
-- in classical terms and in the smooth case -- 
we have for a vector field $E$ ($n$ denotes the exterior unit normal at $\ga$)
\begin{align*}
  E&\in\H{}{\gat}(\rot,\om)
  &&\equi&
           E&\in\L{2}{}(\om),&
              \rot E&\in\L{2}{}(\om),&
              n\times E|_{\gat}&=0,\\
  E&\in\eps^{-1}\H{}{\gan}(\div,\om)
  &&\equi&
           \eps E&\in\L{2}{}(\om),&
                   \div\eps E&\in\L{2}{}(\om),&
                   n\cdot\eps E|_{\gan}&=0.
\end{align*}

Note that \Cref{cptvechomotheo} even holds for bounded weak Lipschitz pairs $(\om,\gat)$.
For exact definitions and notations see \Cref{sec:not},
and for a history of related compact embedding results see, e.g.,
\cite{weck1974a,picard1984a,weber1980a,costabel1990a,witsch1993a,jochmann1997a,picardweckwitsch2001a}
and~\cite{leis1986a}.
The general importance of compact embeddings in a functional analytical setting
(FA-ToolBox) for Hilbert complexes (such as de Rham, elasticity, biharmonic)
is described, e.g., in~\cite{pauly2017a,pauly2019a,pauly2019b,pauly2019c}
and~\cite{paulyzulehner2020a,paulyzulehner2020b,arnoldhu2020a}.

In this paper, we shall generalise \Cref{cptvechomotheo}
to the case of inhomogeneous boundary conditions, i.e., we will show that 
the compact embedding in \Cref{cptvechomotheo} still holds if the space
$$\H{}{\gat}(\rot,\om)\cap\eps^{-1}\H{}{\gan}(\div,\om)$$ 
is replaced by
$$\hH{}{\gat}(\rot,\om)\cap\eps^{-1}\hH{}{\gan}(\div,\om),$$ 
where
in classical terms and in the smooth case
\begin{align*}
  E&\in\hH{}{\gat}(\rot,\om)
  &&\equi&
           E&\in\L{2}{}(\om),&
              \rot E&\in\L{2}{}(\om),&
              n\times E|_{\gat}&\in\L{2}{}(\gat),\\
  E&\in\eps^{-1}\hH{}{\gan}(\div,\om)
  &&\equi&
           \eps E&\in\L{2}{}(\om),&
                   \div\eps E&\in\L{2}{}(\om),&
                   n\cdot\eps E|_{\gan}&\in\L{2}{}(\gan).
\end{align*}

The main result (compact embedding) is formulated in \Cref{cptvecinhomotheo}.
As applications we show in \Cref{estveccor1}
that the compact embedding implies a related Friedrichs/Poincar\'e type estimate,
showing well-posedness of related systems of partial differential equations.
Moreover, in \Cref{divcurllem} we prove that \Cref{cptvecinhomotheo} yields a div-curl lemma.
Note that corresponding results for exterior domains are straight forward
using weighted Sobolev spaces, see~\cite{osterbrinkpauly2019a,osterbrinkpauly2020a}.
Another application is presented in \Cref{sec:maxwell}
where we show that our compact embedding result implies
compact resolvents of the Maxwell operator
with inhomogeneous mixed boundary conditions, even of impedance type.
We finally note in \Cref{sec:wave} that the corresponding result
holds (in the simpler situation) for the impedance wave equation (acoustics) as well.


\section{Notations}%
\label{sec:not}


Throughout this paper, let $\om\subset\reals^{3}$ be an open and bounded
strong Lipschitz domain,
and let $\eps$ be an \emph{admissible} tensor (matrix) field, i.e.,
a symmetric, $\L{\infty}{}$-bounded, and uniformly positive definite 
tensor field $\eps:\om\to\reals^{3\times3}$.
Moreover, let the boundary $\ga$ of $\om$ be decomposed into two relatively open 
and strong Lipschitz subsets $\gat$ and $\gan\coloneqq\ga\setminus\ol{\gat}$
forming the interface $\ol{\gat}\cap\ol{\gan}$ for the mixed boundary conditions.
See~\cite{bauerpaulyschomburg2016a,bauerpaulyschomburg2018a,bauerpaulyschomburg2019a}
for exact definitions. We call $(\om,\gat)$ a bounded strong Lipschitz pair.

The usual Lebesgue and Sobolev Hilbert spaces (of scalar or vector valued fields)
are denoted by $\L{2}{}(\om)$, $\H{1}{}(\om)$, $\H{}{}(\rot,\om)$, $\H{}{}(\div,\om)$,
and by $\H{}{0}(\rot,\om)$ and $\H{}{0}(\div,\om)$ we indicate 
the spaces with vanishing $\rot$ and $\div$, respectively. 
Homogeneous boundary conditions are introduced in the strong sense as closures of respective test fields from
\[
  \C{\infty}{\gat}(\om)\coloneqq
  \big\{\phi|_{\om}\,:\,\phi\in\C{\infty}{}(\reals^{3}),\;
  \mathrm{supp}\,\phi\;\mathrm{compact},\;
  \mathrm{dist}(\mathrm{supp}\,\phi,\gat)>0\big\},
\]
i.e.,
\[
  \H{1}{\gat}(\om)\coloneqq\ol{\C{\infty}{\gat}(\om)}^{\H{1}{}(\om)},
  \quad
  \H{}{\gat}(\rot,\om)\coloneqq\ol{\C{\infty}{\gat}(\om)}^{\H{}{}(\rot,\om)},
  \quad
  \H{}{\gat}(\div,\om)\coloneqq\ol{\C{\infty}{\gat}(\om)}^{\H{}{}(\div,\om)},
\]
and we set $\H{1}{\emptyset}(\om)\coloneqq\H{1}{}(\om)$, $\H{}{\emptyset}(\rot,\om)\coloneqq\H{}{}(\rot,\om)$,
and $\H{}{\emptyset}(\div,\om)\coloneqq\H{}{}(\div,\om)$.
Spaces with vanishing $\rot$ and $\div$
are again denoted by $\H{}{\gat,0}(\rot,\om)$ and $\H{}{\gat,0}(\div,\om)$, respectively. 
Moreover, we introduce the cohomology space of Dirichlet/Neumann fields (generalised harmonic fields)
\[
  \Harm{}{\gat,\gan,\eps}(\om)
  \coloneqq\H{}{\gat,0}(\rot,\om)\cap\eps^{-1}\H{}{\gan,0}(\div,\om).
\]

The $\L{2}{}(\om)$-inner product and norm (of scalar or vector valued $\L{2}{}(\om)$-spaces) will be denoted by 
$\scp{\,\cdot\,}{\,\cdot\,}_{\L{2}{}(\om)}$ and $\norm{\,\cdot\,}_{\L{2}{}(\om)}$, respectively,
and the weighted Lebesgue space $\L{2}{\eps}(\om)$ is defined as $\L{2}{}(\om)$ (of vector fields)
but being equipped with the weighted $\L{2}{}(\om)$-inner product and norm
$\scp{\,\cdot\,}{\,\cdot\,}_{\L{2}{\eps}(\om)}\coloneqq\scp{\eps\,\cdot\,}{\,\cdot\,}_{\L{2}{}(\om)}$
and $\norm{\,\cdot\,}_{\L{2}{\eps}(\om)}$, respectively.
The norms in, e.g., $\H{1}{}(\om)$ and $\H{}{}(\rot,\om)$
are denoted by $\norm{\,\cdot\,}_{\H{1}{}(\om)}$ and $\norm{\,\cdot\,}_{\H{}{}(\rot,\om)}$,
respectively. Orthogonality and orthogonal sum in $\L{2}{}(\om)$ and $\L{2}{\eps}(\om)$
are indicated by $\bot_{\L{2}{}(\om)}$, $\bot_{\L{2}{\eps}(\om)}$, 
and $\oplus_{\L{2}{}(\om)}$, $\oplus_{\L{2}{\eps}(\om)}$, respectively.

Finally, we introduce inhomogeneous tangential and normal $\L{2}{}$-boundary conditions in
\begin{align*}
  \hH{}{\gat}(\rot,\om)
  &\coloneqq\big\{E\in\H{}{}(\rot,\om)\,:\,\tau_{\gat}E\in\L{2}{}(\gat)\big\},\\
  \hH{}{\gan}(\div,\om)
  &\coloneqq\big\{E\in\H{}{}(\div,\om)\,:\,\nu_{\gan}E\in\L{2}{}(\gan)\big\}
\end{align*}
with norms given by, e.g.,
$\norm{E}_{\hH{}{\gat}(\rot,\om)}^2
\coloneqq\norm{E}_{\H{}{}(\rot,\om)}^2
+\norm{\tau_{\gat}E}_{\L{2}{}(\gat)}^2$.
The definitions of the latter Hilbert spaces need some explanations:

\begin{defi}[$\L{2}{}$-traces]
  \label{def:trace}
  \mbox{}
  \begin{enumerate}[label=\myroman{*}]
    \item\label{def:tangential-trace-L2}
      The tangential trace of a vector field $E\in\H{}{}(\rot,\om)$
      is a well-defined tangential vector field $\tau_{\ga}E\in\H{-1/2}{}(\ga)$
      generalising the classical tangential trace $\tau_{\ga}\widetilde{E}=-n\times n\times\widetilde{E}|_{\ga}$
      for smooth vector fields $\widetilde{E}$.
      By the notation $\tau_{\gat}E\in\L{2}{}(\gat)$ we mean, that there exists
      a tangential vector field $E_{\gat}\in\L{2}{}(\gat)$, such that
      for all vector fields $\Phi\in\H{1}{\gan}(\om)$ it holds
      \[
      \scp{\rot\Phi}{E}_{\L{2}{}(\om)}
      -\scp{\Phi}{\rot E}_{\L{2}{}(\om)}
      =\scp{\tau^{\times}_{\gat}\Phi}{E_{\gat}}_{\L{2}{}(\gat)}.
      \]
      Then we set $\tau_{\gat}E\coloneqq E_{\gat}\in\L{2}{}(\gat)$.
      Here and in the following, the twisted tangential trace of the smooth vector field $\Phi$
      is given by the tangential vector field $\tau^{\times}_{\ga}\Phi=n\times\Phi|_{\ga}\in\L{2}{}(\ga)$
      with $\tau^{\times}_{\gan}\Phi=\tau^{\times}_{\ga}\Phi|_{\gan}=0$
      and $\tau^{\times}_{\gat}\Phi=\tau^{\times}_{\ga}\Phi|_{\gat}\in\L{2}{}(\gat)$.
      Note that $\tau_{\gat}E$ is well defined as
      $\tau^{\times}_{\gat}\H{1}{\gan}(\om)$ is dense in
      $\L{2}{t}(\gat)=\big\{v\in\L{2}{}(\gat):n\cdot v=0\big\}$.
    \item\label{def:normal-trace-L2}
      Analogously, the normal trace of a vector field $E\in\H{}{}(\div,\om)$
      is a well-defined function $\nu_{\ga}E\in\H{-1/2}{}(\ga)$
      generalising the classical normal trace $\nu_{\ga}\widetilde{E}=n\cdot\widetilde{E}|_{\ga}$
      for smooth vector fields $\widetilde{E}$.
      Again, by the notation $\nu_{\gan}E\in\L{2}{}(\gan)$ we mean, that
      for all functions $\phi\in\H{1}{\gat}(\om)$ it holds
      \[
      \scp{\grad\phi}{E}_{\L{2}{}(\om)}
      +\scp{\phi}{\div E}_{\L{2}{}(\om)}
      =\scp{\sigma_{\gan}\phi}{\nu_{\gan}E}_{\L{2}{}(\gan)}.
      \]
      Here, the well-known scalar trace of the smooth function $\phi$
      is given by $\sigma_{\ga}\phi=\phi|_{\ga}\in\L{2}{}(\ga)$
      with $\sigma_{\gat}\phi=\sigma_{\ga}\phi|_{\gat}=0$
      and $\sigma_{\gan}\phi=\sigma_{\ga}\phi|_{\gan}\in\L{2}{}(\gan)$.
      Note that $\nu_{\gan}E$ is well defined as
      $\sigma_{\gan}\H{1}{\gat}(\om)$ is dense in $\L{2}{}(\gan)$.
  \end{enumerate}
\end{defi}

\begin{rem}[$\L{2}{}$-traces]
\label{rem:trace}
Analogously to \Cref{def:trace}~\ref{def:tangential-trace-L2} and as
$$\tau^{\times}_{\gat}\widetilde{E}\cdot\tau_{\gat}\widetilde{H}
=(n\times\widetilde{E})\cdot(-n\times n\times\widetilde{H})
=(n\times n\times\widetilde{E})\cdot(n\times\widetilde{H})
=-\tau_{\gat}\widetilde{E}\cdot\tau^{\times}_{\gat}\widetilde{H}$$
holds on $\gat$ for smooth vector fields $\widetilde{E}$, $\widetilde{H}$,
we can define the twisted tangential trace
$\tau^{\times}_{\gat}E\in\L{2}{}(\gat)$ of a vector field $E\in\H{}{}(\rot,\om)$
as well by
\[
  \scp{\rot\Phi}{E}_{\L{2}{}(\om)}
  -\scp{\Phi}{\rot E}_{\L{2}{}(\om)}
  =-\scp{\tau_{\gat}\Phi}{\tau^{\times}_{\gat}E}_{\L{2}{}(\gat)}
\]
for all vector fields $\Phi\in\H{1}{\gan}(\om)$.
\end{rem}


\section{Preliminaries}


In~\cite[Theorem 5.5]{bauerpaulyschomburg2019a}, see~\cite[Theorem 7.4]{bauerpaulyschomburg2018a} for more details and
compare to~\cite{bauerpaulyschomburg2016a},
the following theorem about the existence of regular potentials 
for the rotation with homogeneous mixed boundary conditions has been shown.

\begin{theo}[regular potential for $\rot$ with homogeneous mixed boundary conditions]%
\label{regpottheo}
\begin{align*}
  \H{}{\gan,0}(\div,\om)\cap\Harm{}{\gat,\gan}(\om)^{\bot_{\L{2}{}(\om)}}
  &=\rot\H{}{\gan}(\rot,\om)=\rot\H{1}{\gan}(\om)
\end{align*}
holds together with a regular potential operator mapping 
$\rot\H{}{\gan}(\rot,\om)$ to $\H{1}{\gan}(\om)$ continuously.
In particular, the latter ranges are closed subspaces of $\L{2}{}(\om)$.
\end{theo}

Moreover, we need~\cite[Theorem 5.2]{bauerpaulyschomburg2019a}:

\begin{theo}[Helmholtz decompositions with homogeneous mixed boundary conditions]%
\label{helmtheo}
The ranges $\grad\H{1}{\gat}(\om)$ and $\rot\H{}{\gan}(\rot,\om)$
are closed subspaces of $\L{2}{}(\om)$,
and the $\L{2}{\eps}(\om)$-orthogonal Helmholtz decompositions
\begin{align*}
  \L{2}{\eps}(\om)
  &=\grad\H{1}{\gat}(\om)
    \oplus_{\L{2}{\eps}(\om)}\eps^{-1}\H{}{\gan,0}(\div,\om)\\
  &=\H{}{\gat,0}(\rot,\om)
    \oplus_{\L{2}{\eps}(\om)}\eps^{-1}\rot\H{}{\gan}(\rot,\om)\\
  &=\grad\H{1}{\gat}(\om)
    \oplus_{\L{2}{\eps}(\om)}\Harm{}{\gat,\gan,\eps}(\om)
    \oplus_{\L{2}{\eps}(\om)}\eps^{-1}\rot\H{}{\gan}(\rot,\om)
\end{align*}
hold (with continuous potential operators).
Moreover, $\Harm{}{\gat,\gan,\eps}(\om)$ has finite dimension.
\end{theo}

Combining \Cref{regpottheo} and \Cref{helmtheo} shows immediately the following.

\begin{cor}[regular Helmholtz decomposition with homogeneous mixed boundary conditions]%
\label{reghelmcor}
The $\L{2}{\eps}(\om)$-orthogonal regular Helmholtz decomposition
\begin{align*}
  \L{2}{\eps}(\om)
  &=\grad\H{1}{\gat}(\om)
    \oplus_{\L{2}{\eps}(\om)}\Harm{}{\gat,\gan,\eps}(\om)
    \oplus_{\L{2}{\eps}(\om)}\eps^{-1}\rot\H{1}{\gan}(\om)
\end{align*}
holds (with continuous potential operators)
and $\Harm{}{\gat,\gan,\eps}(\om)$ has finite dimension.
More precisely, any $E\in\L{2}{\eps}(\om)$ may be $\L{2}{\eps}(\om)$-orthogonally (and regularly) decomposed into
\[
  E=\grad u_{\grad}+E_{\mcH}+\eps^{-1}\rot E_{\rot}
\]
with $u_{\grad}\in\H{1}{\gat}(\om)$, $E_{\rot}\in\H{1}{\gan}(\om)$, and 
$E_{\mcH}\in\Harm{}{\gat,\gan,\eps}(\om)$, and there exists a constant $c>0$,
independent of $E,u_{\grad},E_{\mcH},E_{\rot}$, such that
\begin{align*}
  \norm{E_{\mcH}}_{\L{2}{\eps}(\om)}
  &\leq\norm{E}_{\L{2}{\eps}(\om)},\\
  c\norm{u_{\grad}}_{\H{1}{\gat}(\om)}
  \leq\norm{\grad u_{\grad}}_{\L{2}{\eps}(\om)}
  &\leq\norm{E}_{\L{2}{\eps}(\om)},\\
  c\norm{E_{\rot}}_{\H{1}{\gan}(\om)}
  \leq\norm{\eps^{-1}\rot E_{\rot}}_{\L{2}{\eps}(\om)}
  &\leq\norm{E}_{\L{2}{\eps}(\om)}.
\end{align*}
\end{cor}


\section{Compact Embeddings}


Our main result reads as follows:

\begin{theo}[compact embedding for vector fields with inhomogeneous mixed boundary conditions]%
\label{cptvecinhomotheo}
\[
  \hH{}{\gat}(\rot,\om)\cap\eps^{-1}\hH{}{\gan}(\div,\om)
  \cpt
  \L{2}{}(\om).
\]
\end{theo}

\begin{proof}
Let $(E_{\ell})$ be a bounded sequence in $\hH{}{\gat}(\rot,\om)\cap\eps^{-1}\hH{}{\gan}(\div,\om)$.
By the Helmholtz decomposition in \Cref{reghelmcor} we
$\L{2}{\eps}(\om)$-orthogonally and regularly decompose
\[
  E_{\ell}=\grad u_{\grad,\ell}+E_{\mcH,\ell}+\eps^{-1}\rot E_{\rot,\ell}
\]
with $u_{\grad,\ell}\in\H{1}{\gat}(\om)$, $E_{\rot,\ell}\in\H{1}{\gan}(\om)$, and
$E_{\mcH,\ell}\in\Harm{}{\gat,\gan,\eps}(\om)$, and there exists a constant $c>0$ such that
independent of $E_{\dots}$ and for all $\ell$
\[
  \norm{u_{\grad,\ell}}_{\H{1}{\gat}(\om)}
  +\norm{E_{\mcH,\ell}}_{\L{2}{\eps}(\om)}
  +\norm{E_{\rot,\ell}}_{\H{1}{\gan}(\om)}
  \leq c\norm{E_{\ell}}_{\L{2}{\eps}(\om)}.
\]
As $\Harm{}{\gat,\gan,\eps}(\om)$ is finite dimensional 
we may assume (after extracting a subsequence) that $E_{\mcH,\ell}$ converges strongly in $\L{2}{\eps}(\om)$.
Since $\H{1}{}(\om)\cpt\L{2}{}(\om)$ by Rellich's selection theorem, we may assume 
that also the regular potentials $u_{\grad,\ell}$ and $E_{\rot,\ell}$ 
converge strongly in $\L{2}{}(\om)$.
Moreover, $u_{\grad,\ell}|_{\ga}$ and $E_{\rot,\ell}|_{\ga}$ are bounded in
$\H{1/2}{}{}(\ga)$ by the (scalar) trace theorem, and thus we may assume 
by the compact embedding $\H{1/2}{}{}(\ga)\cpt\L{2}{}(\ga)$ that 
$u_{\grad,\ell}|_{\ga}$ and $E_{\rot,\ell}|_{\ga}$ converge strongly in $\L{2}{}{}(\ga)$.
In particular, $u_{\grad,\ell}|_{\gan}$ and $E_{\rot,\ell}|_{\gat}$ 
converge strongly in $\L{2}{}{}(\gan)$ and $\L{2}{}{}(\gat)$, respectively.
After all this successively taking subsequences we obtain 
(using $\L{2}{\eps}(\om)$-orthogonality and the definition 
of the $\L{2}{}(\gan)$-traces of $\nu_{\gan}\eps E_{\ell}$
and the $\L{2}{}(\gat)$-traces of $\tau_{\gat}E_{\ell}$ from \Cref{def:trace})
\begin{align*}
  \MoveEqLeft[4]
  \norm[\big]{\grad(u_{\grad,\ell}-u_{\grad,k})}_{\L{2}{\eps}(\om)}^2 \\
  &=\scp[\big]{\grad(u_{\grad,\ell}-u_{\grad,k})}{E_{\ell}-E_{k}}_{\L{2}{\eps}(\om)}\\
  &=-\scp[\big]{u_{\grad,\ell}-u_{\grad,k}}{\div\eps(E_{\ell}-E_{k})}_{\L{2}{}(\om)}
    +\scp[\big]{\sigma_{\gan}(u_{\grad,\ell}-u_{\grad,k})}{\nu_{\gan}\eps(E_{\ell}-E_{k})}_{\L{2}{}(\gan)}\\
  &\leq c\norm{u_{\grad,\ell}-u_{\grad,k}}_{\L{2}{}(\om)}
    +c\norm[\big]{(u_{\grad,\ell}-u_{\grad,k})|_{\gan}}_{\L{2}{}(\gan)}
    \to0
    \intertext{and}
  \MoveEqLeft[4]
    \norm[\big]{\eps^{-1}\rot(E_{\rot,\ell}-E_{\rot,k})}_{\L{2}{\eps}(\om)}^2 \\
  &=\scp[\big]{\eps^{-1}\rot(E_{\rot,\ell}-E_{\rot,k})}{E_{\ell}-E_{k}}_{\L{2}{\eps}(\om)}\\
  &=\scp[\big]{E_{\rot,\ell}-E_{\rot,k}}{\rot(E_{\ell}-E_{k})}_{\L{2}{}(\om)}
    +\scp[\big]{\tau^{\times}_{\gat}(E_{\rot,\ell}-E_{\rot,k})}{\tau_{\gat}(E_{\ell}-E_{k})}_{\L{2}{}(\gat)}\\
  &\leq c\norm{E_{\rot,\ell}-E_{\rot,k}}_{\L{2}{}(\om)}
    +c\norm[\big]{(E_{\rot,\ell}-E_{\rot,k})|_{\gat}}_{\L{2}{}(\gat)}
\to0.
\end{align*}
Hence, $(E_{\ell})$ contains a strongly $\L{2}{\eps}(\om)$-convergent
(and thus strongly $\L{2}{}(\om)$-convergent) subsequence.
\end{proof}

\begin{rem}[compact embedding for vector fields with inhomogeneous mixed boundary conditions]%
\label{cptdiffinhomorem}
\Cref{cptvecinhomotheo} even holds for weaker boundary data. For this, let $0\leq s<1/2$.
Taking into account the compact embedding $\H{1/2}{}{}(\ga)\cpt\H{s}{}{}(\ga)$
and looking at the latter proof, we see that
\[
  \big\{E\in\H{}{}(\rot,\om)\,:\,\tau_{\gat}E\in\H{-s}{}{}(\gat)\big\}
  \cap
  \big\{E\in\eps^{-1}\H{}{}(\div,\om)\,:\,\nu_{\gan}\eps E\in\H{-s}{}{}(\gan)\big\}
  \cpt
  \L{2}{}(\om).
\]
\end{rem}


\section{Applications}
\label{sec:applications}


\subsection{Friedrichs/Poincar\'e Type Estimates}


A first application is the following estimate:

\begin{theo}[Friedrichs/Poincar\'e type estimate for vector fields with inhomogeneous mixed boundary conditions]%
\label{estveccor1}
There exists a positive constant $c$ such that for all vector fields $E$ in 
$\hH{}{\gat}(\rot,\om)\cap\eps^{-1}\hH{}{\gan}(\div,\om)\cap\Harm{}{\gat,\gan,\eps}(\om)^{\bot_{\L{2}{\eps}(\om)}}$
it holds
\[
  c\norm{E}_{\L{2}{\eps}(\om)}
  \leq\norm{\rot E}_{\L{2}{}(\om)}
  +\norm{\div\eps E}_{\L{2}{}(\om)}
  +\norm{\tau_{\gat}E}_{\L{2}{}(\gat)}
  +\norm{\nu_{\gan}\eps E}_{\L{2}{}(\gan)}.
\]
\end{theo}

\begin{proof}
For a proof we use a standard compactness argument using \Cref{cptvecinhomotheo}.
If the estimate was wrong, then there exists a sequence 
$(E_{\ell})\in\hH{}{\gat}(\rot,\om)\cap\eps^{-1}\hH{}{\gan}(\div,\om)\cap\Harm{}{\gat,\gan,\eps}(\om)^{\bot_{\L{2}{\eps}(\om)}}$
with $\norm{E_{\ell}}_{\L{2}{\eps}(\om)}=1$ and 
\[
  \norm{\rot E_{\ell}}_{\L{2}{}(\om)}
  +\norm{\div\eps E_{\ell}}_{\L{2}{}(\om)}
  +\norm{\tau_{\gat}E_{\ell}}_{\L{2}{}(\gat)}
  +\norm{\nu_{\gan}\eps E_{\ell}}_{\L{2}{}(\gan)}\to0.
\]
Thus, by \Cref{cptvecinhomotheo} (after extracting a subsequence)
\[
E_{\ell}\to E
\quad\mathrm{in}\quad
\hH{}{\gat}(\rot,\om)\cap\eps^{-1}\hH{}{\gan}(\div,\om)\cap\Harm{}{\gat,\gan,\eps}(\om)^{\bot_{\L{2}{\eps}(\om)}}
\qquad\mathrm{(strongly)}
\]
and $\rot E=0$ and $\div\eps E=0$ (by testing).
Moreover, for all $\Phi\in\C{\infty}{\gan}(\om)$ and for all
$\phi\in\C{\infty}{\gat}(\om)$
\begin{align*}
  \scp{\rot\Phi}{E_{\ell}}_{\L{2}{}(\om)}
  -\scp{\Phi}{\rot E_{\ell}}_{\L{2}{}(\om)}
  &=\scp{\tau^{\times}_{\gat}\Phi}{\tau_{\gat}E_{\ell}}_{\L{2}{}(\gat)}
    \leq c\norm{\tau_{\gat}E_{\ell}}_{\L{2}{}(\gat)}
    \to 0,\\
  \scp{\grad\phi}{\eps E_{\ell}}_{\L{2}{}(\om)}
  +\scp{\phi}{\div\eps E_{\ell}}_{\L{2}{}(\om)}
  &=\scp{\sigma_{\gan}\phi}{\nu_{\gan}\eps E_{\ell}}_{\L{2}{}(\gan)}
    \leq c\norm{\nu_{\gan}\eps E_{\ell}}_{\L{2}{}(\gan)}
    \to 0,
\end{align*}
cf.~\Cref{def:trace}, implying 
\[
  \scp{\rot\Phi}{E}_{\L{2}{}(\om)}
  =\scp{\grad\phi}{\eps E}_{\L{2}{}(\om)}
  =0.
\]
Hence, 
$E\in\H{}{\gat,0}(\rot,\om)\cap\eps^{-1}\H{}{\gan,0}(\div,\om)=\Harm{}{\gat,\gan,\eps}(\om)$ 
by~\cite[Theorem 4.7]{bauerpaulyschomburg2019a} (weak and strong homogeneous boundary conditions coincide).
This shows $E=0$ as $E\,\bot_{\L{2}{\eps}(\om)}\,\Harm{}{\gat,\gan,\eps}(\om)$, in contradiction to 
$1=\norm{E_{\ell}}_{\L{2}{\eps}(\om)}\to\norm{E}_{\L{2}{\eps}(\om)}=0$.
\end{proof}

\begin{rem}[Friedrichs/Poincar\'e type estimate for vector fields with inhomogeneous mixed boundary conditions]%
\label{estdiffinhomorem}
As in \Cref{cptdiffinhomorem} there are corresponding generalised
Friedrichs/Poincar\'e type estimates for weaker boundary data, where
the $\L{2}{}{}(\ga_{0/1})$-spaces and norms are 
replaced by $\H{-s}{}{}(\ga_{0/1})$-spaces and norms.
\end{rem}


\subsection{A div-curl Lemma}


Another immediate consequence is a div-curl-lemma.

\begin{theo}[div-curl lemma for vector fields with inhomogeneous mixed boundary conditions]%
\label{divcurllem}
Let $(E_{n})$ and $(H_{n})$ be bounded sequences
in $\hH{}{\gat}(\rot,\om)$ and $\hH{}{\gan}(\div,\om)$, respectively.
Then there exist $E\in\hH{}{\gat}(\rot,\om)$ and $H\in\hH{}{\gan}(\div,\om)$
as well as subsequences, again denoted by $(E_{n})$ and $(H_{n})$, such that
$E_{n}\rightharpoonup E$ in $\hH{}{\gat}(\rot,\om)$
and $H_{n}\rightharpoonup H$ in $\hH{}{\gan}(\div,\om)$ as well as 
\[
\scp{E_{n}}{H_{n}}_{\L{2}{}(\om)}
\to\scp{E}{H}_{\L{2}{}(\om)}.
\]
\end{theo}

\begin{proof}
We follow in closed lines the proof of \cite[Theorem 3.1]{pauly2019c}.
Let $(E_{n})$ and $(H_{n})$ be as stated.
First, we pick subsequences, again denoted by $(E_{n})$ and $(H_{n})$,
and $E$ and $H$, such that
$E_{n}\rightharpoonup E$ in $\hH{}{\gat}(\rot,\om)$
and $H_{n}\rightharpoonup H$ in $\hH{}{\gan}(\div,\om)$.
In particular, 
\begin{align}
\label{weakconvtrace}
\nu_{\gan}H_{n}\rightharpoonup\nu_{\gan}H
\quad\text{in}\quad\L{2}{}(\gan).
\end{align}

To see~\eqref{weakconvtrace}, let $\nu_{\gan}H_{n}\rightharpoonup H_{\gan}$ in $\L{2}{}(\gan)$.
Since for all $\phi\in\H{1}{\gat}(\om)$
\begin{align*}
\scp{\sigma_{\gan}\phi}{H_{\gan}}_{\L{2}{}(\gan)}
\leftarrow\scp{\sigma_{\gan}\phi}{\nu_{\gan}H_{n}}_{\L{2}{}(\gan)}
=&\;\scp{\grad\phi}{H_{n}}_{\L{2}{}(\om)}
+\scp{\phi}{\div H_{n}}_{\L{2}{}(\om)}\\
\to&\;\scp{\grad\phi}{H}_{\L{2}{}(\om)}
+\scp{\phi}{\div H}_{\L{2}{}(\om)},
\end{align*}
we get $H\in\hH{}{\gan}(\div,\om)$ and $\nu_{\gan}H=H_{\gan}$.
Moreover,
$\scp{\sigma_{\gan}\phi}{\nu_{\gan}H_{n}}_{\L{2}{}(\gan)}
\to\scp{\sigma_{\gan}\phi}{\nu_{\gan}H}_{\L{2}{}(\gan)}$.
As $\sigma_{\gan}\H{1}{\gat}(\om)$ is dense in $\L{2}{}(\gan)$
and $\big(\scp{\,\cdot\,}{\nu_{\gan}H_{n}}_{\L{2}{}(\gan)}\big)$
is uniformly bounded with respect to $n$ we obtain \eqref{weakconvtrace}.

By \Cref{helmtheo} we have the orthogonal Helmholtz decomposition
\[
\hH{}{\gat}(\rot,\om)\ni E_{n}=\grad u_{n}+\widetilde{E}_{n}
\]
with $u_{n}\in\H{1}{\gat}(\om)$
and $\widetilde{E}_{n}\in\hH{}{\gat}(\rot,\om)\cap\H{}{\gan,0}(\div,\om)$
as $\grad\H{1}{\gat}(\om)\subset\H{}{\gat,0}(\rot,\om)\subset\hH{}{\gat}(\rot,\om)$.
By orthogonality and the Friedrichs/Poincar\'e estimate,
$(u_{n})$ is bounded in $\H{1}{\gat}(\om)$ and hence contains 
a strongly $\L{2}{}(\om)$-convergent subsequence, again denoted by $(u_{n})$.
(For $\gat=\emptyset$ we may have to add a constant to each $u_{n}$.)
Moreover, as $(u_{n}|_{\ga})$ is bounded in
$\H{1/2}{}{}(\ga)\cpt\L{2}{}(\ga)$ we may assume that 
$(u_{n}|_{\ga})$ converges strongly in $\L{2}{}(\ga)$.
In particular, $(\sigma_{\gan}u_{n})=(u_{n}|_{\gan})$ converges strongly in $\L{2}{}(\gan)$.
The sequence $(\widetilde{E}_{n})$ is bounded in $\hH{}{\gat}(\rot,\om)\cap\H{}{\gan,0}(\div,\om)$
by orthogonality and since $\rot\widetilde{E}_{n}=\rot E_{n}$
and $\tau_{\gat}\widetilde{E}_{n}=\tau_{\gat}E_{n}$.
\Cref{cptvecinhomotheo} yields a strongly $\L{2}{}(\om)$-convergent subsequence, 
again denoted by $(\widetilde{E}_{n})$.
Hence, there exist $u\in\H{1}{\gat}(\om)$ and 
$\widetilde{E}\in\hH{}{\gat}(\rot,\om)\cap\H{}{\gan,0}(\div,\om)$
such that $u_{n}\rightharpoonup u$ in $\H{1}{\gat}(\om)$ and 
$u_{n}\to u$ in $\L{2}{}(\om)$ and 
$\sigma_{\gan}u_{n}\to\sigma_{\gan}u$ in $\L{2}{}(\gan)$
as well as $\widetilde{E}_{n}\rightharpoonup\widetilde{E}$ in 
$\hH{}{\gat}(\rot,\om)\cap\H{}{\gan,0}(\div,\om)$
and $\widetilde{E}_{n}\to\widetilde{E}$ in $\L{2}{}(\om)$.
Finally, we compute
\begin{align*}
  \scp{E_{n}}{H_{n}}_{\L{2}{}(\om)}
  &= \scp{\grad u_{n}}{H_{n}}_{\L{2}{}(\om)}
    +\scp{\widetilde{E}_{n}}{H_{n}}_{\L{2}{}(\om)}\\
  &= -\scp{u_{n}}{\div H_{n}}_{\L{2}{}(\om)}
    +\scp{\sigma_{\gan}u_{n}}{\nu_{\gan}H_{n}}_{\L{2}{}(\gan)}
    +\scp{\widetilde{E}_{n}}{H_{n}}_{\L{2}{}(\om)}\\
  &\to -\scp{u}{\div H}_{\L{2}{}(\om)}
    +\scp{\sigma_{\gan}u}{\nu_{\gan}H}_{\L{2}{}(\gan)}
    +\scp{\widetilde{E}}{H}_{\L{2}{}(\om)}\\
  &= \scp{\grad u}{\div H}_{\L{2}{}(\om)}
    +\scp{\widetilde{E}}{H}_{\L{2}{}(\om)}
    =\scp{E}{H}_{\L{2}{}(\om)},
\end{align*}
since indeed $E=\grad u+\widetilde{E}$ holds by the weak convergence.
\end{proof}

\begin{rem}[div-curl lemma for vector fields with inhomogeneous mixed boundary conditions]%
\label{divcurllemrem}
As in \Cref{cptdiffinhomorem} and \Cref{estdiffinhomorem} there are corresponding generalised
div-curl lemmas for weaker boundary data, where
the $\L{2}{}{}(\ga_{0/1})$-spaces and norms are 
replaced by $\H{-s}{}{}(\ga_{0/1})$-spaces and norms.
\end{rem}


\subsection{Maxwell's Equations with Mixed Impedance Type Boundary Conditions}
\label{sec:maxwell}


Let $\eps$, $\mu$ be admissible and time-independent matrix fields,
and let $T,k\in\reals_{+}$.
In $\tdom\times\om$ with $\tdom\coloneqq(0,T)$ we consider Maxwell's equations
with mixed tangential and impedance boundary conditions
\begin{subequations}
\label{eq:maxwell}
\begin{align}
\label{eq:maxwell-ampere-law}
\p_{t}E-\eps^{-1}\rot H&=F\in\L{2}{}\big(\tdom,\L{2}{\eps}(\om)\big),
&
&\text{(Amp\`ere/Maxwell law)}\\
\label{eq:maxwell-faraday-law}
\p_{t}H+\mu^{-1}\rot E&=G\in\L{2}{}\big(\tdom,\L{2}{\mu}(\om)\big),
&&\text{(Faraday/Maxwell law)}
\\
\label{eq:maxwell-gauss-law}
\div \eps E &= \rho, && \text{(Gau{\ss} law)}
\\
\label{eq:maxwell-magnetic-gauss-law}
\div \mu H &= 0, && \text{(Gau{\ss} law for magnetism)}
\\
\label{eq:maxwell-conductor}
\tau_{\gat}E &= 0,
&&\text{(perfect conductor bc)}\\
\label{eq:maxwell-normal-trace}
\nu_{\gat}H &= f,
&&\text{(normal trace bc)}\\
\label{maxwell-impedance-bc}
\tau_{\gan}E+k\tau^{\times}_{\gan}H&=0,
&&\text{(impedance bc)}\\
\label{maxwell-electric-iv}
E(0)&=E_{0}\in\L{2}{\eps}(\om),
&&\text{(electric initial value)}\\
\label{maxwell-magnetic-iv}
H(0)&=H_{0}\in\L{2}{\mu}(\om).
&&\text{(magnetic initial value)}
\end{align}
\end{subequations}

Note that the impedance boundary condition, also called Leontovich boundary condition,
is of Robin type and that the impedance is given by $\lambda=1/k=\sqrt{\eps/\mu}$,
if $\eps,\mu\in\reals_{+}$.

Despite of other recent and very powerful approaches
such as the concept of ``evolutionary equations'', see the pioneering work of Rainer Picard,
e.g., \cite{picard2009a,picard2020a},
one can use classical semigroup theory for solving the Maxwell system~\eqref{eq:maxwell}.

We will split the system~\eqref{eq:maxwell} into two static systems and a dynamic system.
For simplicity we set $\eps = \mu = 1$ and $F = G = 0$. 
The static systems are
\begin{subequations}
  \label{eq:maxwell-static}
  \begin{align}
    \rot E &= 0, & \rot H &= 0, \\
    \div E &= \rho, & \div H &= 0, \\
    \tau_{\gat} E &= 0, & \nu_{\gat} H &= f, \\
    \tau_{\gan} E &= -k g, & \tau_{\gan}^{\times} H &= g,
  \end{align}
\end{subequations}
where $g$ is any suitable tangential vector field in $\L{2}{}(\gan)$. 
For simplicity we put $g=0$, then these two systems are 
solvable by~\cite[Theorem~5.6]{bauerpaulyschomburg2016a}. 
However, the same result also gives conditions for which $g\neq 0$ this system is solvable.
The dynamic system is
\begin{subequations}
  \label{eq:maxwell-dynamic}
  \begin{align}
    \label{eq:maxwell-dynamic-ampere}
    \p_{t} E &= \rot H, \\
    \label{eq:maxwell-dynamic-faraday}
    \p_{t} H &= -\rot E, \\
    \label{eq:maxwell-dynamic-gauss}
    \div E &= 0, \\
    \label{eq:maxwell-dynamic-gauss-magnetism}
    \div H &= 0, \\
    \nu_{\gat} H &= 0, \\
    \label{eq:maxwell-dynamic-conductor}
    \tau_{\gat} E &= 0, \\
    \label{eq:maxwell-dynamic-impedance}
    \tau_{\gan} E + k \tau_{\gan}^{\times} H &= 0.
  \end{align}
\end{subequations}
The initial conditions for the dynamic system are $E(0) = E_{0} - E_{\mathrm{stat}}$ and $H(0) = H_{0} - H_{\mathrm{stat}}$, where $E_{\mathrm{stat}}$ and $H_{\mathrm{stat}}$ are the solutions of the two static systems~\eqref{eq:maxwell-static}.
We can write~\eqref{eq:maxwell-dynamic-ampere} and~\eqref{eq:maxwell-dynamic-faraday} as
\begin{align*}
  \p_t
  \begin{bmatrix}    E \\ H  \end{bmatrix}
  =
  \underbrace{
  \begin{bmatrix}    0 & \rot \\    -\rot & 0  \end{bmatrix}
  }_{\phantom{\A_{0}}\eqqcolon \A_{0}}
  \begin{bmatrix}    E \\ H  \end{bmatrix},
\end{align*}
and the boundary conditions~\eqref{eq:maxwell-dynamic-conductor}
and~\eqref{eq:maxwell-dynamic-impedance}
shall be covered by the domain of $\A_{0}$:
\[
  \dom \A_{0} \coloneqq
  \left\{
  	(E,H)    \in
    \hH{}{\ga}(\rot,\Omega) \times \hH{}{\gan}(\rot,\Omega)
    \;\middle|\;
    \tau_{\gat} E = 0,\;
    \tau_{\gan} E + k \tau_{\gan}^{\times} H = 0
  \right\}.
\]
Here, we did ignore the equations $\div E = 0$, $\div H = 0$ and $\nu_{\gat} H = 0$.
However, $\A_{0}$ is a generator of a $\C{}{0}$-semigroup by~\cite[Example~8.10]{skrepek-phs-1}
or~\cite[Section 5]{weiss-staffans-maxwell},
where the input function is $u=0$.
(In these sources they regard boundary control systems and system nodes, respectively.
One condition of those concepts is that the system with $u=0$ 
is described by a generator of a $\C{}{0}$-semigroup).
The next lemma provides a tool to respect the remaining
conditions of~\eqref{eq:maxwell-dynamic} as well.

\begin{lem}\label{le:invariant-subspace}%
  Let $\T(\cdot)$ be a $\C{}{0}$-semigroup on a Banach space $X$, 
  and let $\A$ be its generator.
  Then every subspace $V \supseteq \ran \A$ is invariant under $\T(\cdot)$.
  Moreover, $\A\big\vert_{V}$ generates the strongly continuous semigroup
  $\T_{V}(\cdot) \coloneqq \T(\cdot)\big\vert_{V}$, 
  if $V$ is additionally closed in $X$.
\end{lem}

\begin{proof}
  Let $t \geq 0$ and let $x\in V$. 
  Then $\ran\A\ni\A\int_{0}^{t}\T(s)x\dx[s]=\T(t)x-x$
  and hence $\T(t)x \in V$.
  The remaining assertion follows from~\cite[Chapter II, Section 2.3]{engel-nagel}.
\end{proof}

Therefore, it is left to show that the remaining conditions 
establish a closed and invariant subspace under the semigroup $\T_{0}$ generated 
by $\A_{0}$ or contains $\ran \A_{0}$.
Note that by \Cref{regpottheo}
\begin{align*}
S&\coloneqq\left\{(E,H)\;\middle|\;\div E = 0,\,\div H=0,\,\nu_{\gat}H=0\right\}\\
&\;=\H{}{0}(\div,\Omega)\times\H{}{\gat,0}(\div,\Omega)\\ 
&\;=\big(\rot \H{}{}(\rot,\Omega) \times \rot \H{}{\gat}(\rot,\Omega)\big)
\oplus
\big(\Harm{}{\ga,\emptyset}(\Omega) \times \Harm{}{\gan,\gat}(\Omega)\big).
\end{align*}
This space is closed as the intersection of kernels of closed operators.
Clearly, $\Harm{}{\ga,\emptyset}(\Omega) \times \Harm{}{\gan,\gat}(\Omega)$ 
is invariant under $\T_{0}$, since every
$(E,H)\in\Harm{}{\ga,\emptyset}(\Omega) \times \Harm{}{\gan,\gat}(\Omega)$
is a constant in time solution of the system~\eqref{eq:maxwell-dynamic}, i.e.,
\[
  \T_{0}(t)
  \begin{bmatrix}    E \\ H  \end{bmatrix}
  =
  \begin{bmatrix}    E \\ H  \end{bmatrix}.
\]
By
\[
  \rot \H{}{}(\rot,\Omega) \times \rot \H{}{\gat}(\rot,\Omega)
  =
  \begin{bmatrix}    0 & \rot \\    -\rot & 0  \end{bmatrix}
  \big(\H{}{\gat}(\rot,\Omega) \times \H{}{}(\rot,\Omega)\big)
  \supseteq \ran \A_{0}
\]
and \Cref{le:invariant-subspace} we have that also
$\rot \H{}{}(\rot,\Omega) \times \rot \H{}{\gat}(\rot,\Omega)$
is invariant under $\T_{0}$.
Hence, \Cref{le:invariant-subspace} and \Cref{cptvecinhomotheo} imply the next theorem.

\begin{theo}
  $\A \coloneqq \A_{0}\big\vert_{S}$ is a generator of a $\C{}{0}$-semigroup and
  \[
    \dom \A \subseteq 
    \big(\hH{}{\ga}(\rot,\Omega) \cap \H{}{}(\div,\Omega)\big) 
    \times 
    \big(\hH{}{\gan}(\rot,\Omega) \cap \H{}{\gat}(\div,\Omega)\big)
    \cpt \L{2}{}(\Omega).
  \]
  Consequently, every resolvent operator of $\A$ is compact.
\end{theo}

If $\Harm{}{\ga,\emptyset}(\Omega)=\{0\}$ and
$\Harm{}{\gan,\gat}(\Omega)=\{0\}$,
then $0$ is in the resolvent set of $\A$ and $\A^{-1}$ is compact. Alternatively, 
we can further restrict $\A$ to 
$\Harm{}{\ga,\emptyset}(\Omega)^{\perp_{\L{2}{}(\om)}} 
\times \Harm{}{\gan,\gat}(\Omega)^{\perp_{\L{2}{}(\om)}}$. 
This would also match our separation of static solutions and dynamic solutions,
since solutions with initial condition in 
$\Harm{}{\ga,\emptyset}(\Omega) \times \Harm{}{\gan,\gat}(\Omega)$ are constant in time.


\subsection{Wave Equation with Mixed Impedance Type Boundary Conditions}%
\label{sec:wave}


For the scalar wave equation the situation is even simpler 
since traces of $\H{1}{}(\om)$-functions already belong to $\L{2}{}(\ga)$,
even to $\H{1/2}{}(\ga)$.
In $\tdom\times\om$ we consider the wave equation in first order form
(linear acoustics) with mixed scalar and impedance boundary conditions
\begin{align*}
  \p_{t}w - \div v&=0,\\
  \p_{t}v - \grad w&=0,\\
  \sigma_{\gat}w&=0,\\
  \sigma_{\gan}w+k\nu_{\gan}v&=0,\\
  w(0)&=w_{0}\in\L{2}{}(\om),\\
  v(0)&=v_{0}\in\L{2}{}(\om).
\end{align*}
We write the system as
\[
  \p_{t}
  \begin{bmatrix}    w \\ v  \end{bmatrix}
  =
  \underbrace{
  \begin{bmatrix}    0 & \div \\    \grad & 0  \end{bmatrix}
  }_{\phantom{\A_{0}}\eqqcolon \A_{0}}
  \begin{bmatrix}    w \\ v  \end{bmatrix}
\]
with
\[
  \dom \A_{0} \coloneqq
  \big\{
  (w,v)
  \in \H{1}{\gat}(\Omega) \times \hH{}{\gan}(\div,\Omega)
  \;\big|\;
  \sigma_{\gan}w + k\nu_{\gan} = 0
  \big\}.
\]
As before, by~\cite[Theorem~4.4]{kurula-zwart-2015} or~\cite[Example~8.9]{skrepek-phs-1}, 
$\A_{0}$ is a generator of $\C{}{0}$-semigroup. 
Again, we want to separate the static solutions from the dynamic system. The static solutions are given by $\ker \A_{0}$, which can be characterise by
\[
  \ker \A_{0} = \{0\} \times \H{}{\gan,0}(\div,\Omega),
\]
where we assumed $\gat \neq \emptyset$, otherwise the first component can also be constant and the second component would be in $\hH{}{\gan,0}(\div,\Omega)$.
By \Cref{helmtheo}, the orthogonal complement of $\ker \A_{0}$ is
\[
  S \coloneqq \L{2}{}(\Omega) \times \grad \H{1}{\gat}(\Omega).
\]
Note that $S$ contains $\ran \A_{0}$ and is therefore (by \Cref{le:invariant-subspace}) an invariant subspace under the semigroup generated by $\A_{0}$. Moreover, note that $\grad \H{1}{\gat}(\Omega) \subseteq \H{}{\gat,0}(\rot,\Omega)$ and that $S$ is closed. 
Hence, \Cref{le:invariant-subspace} and \Cref{cptvecinhomotheo} imply the next theorem.

\begin{theo}
  $\A \coloneqq \A_{0}\big\vert_{S}$ is a generator of $\C{}{0}$-semigroup and
  \[
    \dom \A    \subseteq
    \H{1}{\gat}(\Omega) \times \big(\hH{}{\gan}(\div,\Omega) \cap \H{}{\gat,0}(\rot,\Omega) \big)
    \cpt \L{2}{}(\Omega).
  \]
  Consequently, every resolvent operator of $\A$ is compact.
\end{theo}

Alternatively, we can also regard the classical formulation of the wave equation 
and see that it is necessary for the second component in our formulation to be 
in $\grad \H{1}{\gat}(\Omega)$, if we want the solutions to correspond.


\bibliographystyle{plain} 
\bibliography{/Users/paule/GoogleDriveData/Tex/input/bibTex/ps}

\vspace*{5mm}\hrule\vspace*{3mm}

\end{document}